\documentclass{amsart}
\usepackage{amssymb}
\usepackage{latexsym}
\usepackage{euscript}

\usepackage{todonotes} 
\usepackage{theoremref}

\usepackage{enumerate} 

\def\sD{{\mathfrak D}}   \def\sE{{\mathfrak E}}   \def\sF{{\mathfrak F}}
   \def\sH{{\mathfrak H}}   
   \def\sK{{\mathfrak K}}

\def\st{{\mathfrak t}}

      \def\dC{{\mathbb C}}

\def\s{{\rm s}}
\def\r{{\rm r}}
\def\I{{\rm i}}

\def\bB{{\boldsymbol B}}

\def\wt#1{{{\widetilde #1} }}

\def\wh#1{{{\widehat #1} }}

\def\bm\chi{\mbox{\boldmath$\chi$}}
\def\half{{\frac{1}{2}}}

\def\RE{{\rm Re\,}}
\def\IM{{\rm Im\,}}
\def\ker{{\rm ker\,}}
\def\ran{{\rm ran\,}}
\def\cran{{\rm \overline{ran}\,}}
\def\dom{{\rm dom\,}}
\def\mul{{\rm mul\,}}

\def\cdom{{\rm d\overline{om}\,}}
\def\clos{{\rm clos\,}}

\def\diag{{\rm diag\,}}

\let\xker=\ker \def\ker{{\xker\,}}

\def\uphar{{\upharpoonright\,}}

\def\op{{\rm s}}

\DeclareMathOperator{\hoplus}{\, \widehat \oplus \,}

\newtheorem{theorem}{Theorem}[section]

\newtheorem{corollary}[theorem]{Corollary}
\newtheorem{lemma}[theorem]{Lemma}

\theoremstyle{definition}
\newtheorem{example}[theorem]{Example}
\newtheorem{remark}[theorem]{Remark}

\numberwithin{equation}{section}

\begin{document}

\title
{A class of sectorial relations and the associated closed forms}

\author{S.~Hassi}
\address{
Department of Mathematics and Statistics \\
University of Vaasa \\
P.O. Box 700, 65101 Vaasa \\
Finland}
\email{sha@uwasa.fi}

\author{H.S.V.~de~Snoo}
\address{
Bernoulli Institute for Mathematics, Computer Science
and Artificial Intelligence\\
University of Groningen \\
P.O. Box 407, 9700 AK Groningen \\
Nederland}
\email{hsvdesnoo@gmail.com}

\dedicatory{Dedicated to V.E. Katsnelson
on the occasion of his 75th birthday}

\thanks{The second author thanks the University of Vaasa for its hospitality during the
preparation of this work.}




\begin{abstract}
Let $T$ be a closed linear relation from a Hilbert space $\sH$
to a Hilbert space $\sK$ and let $B \in \mathbf{B}(\sK)$ be selfadjoint.
 It will be shown that the relation $T^{*}(I+iB)T$ is maximal sectorial
via a matrix decomposition of $B$ with respect to the orthogonal
decomposition $\sH=\cdom T^* \oplus \mul T$.
This leads to an explicit expression of the corresponding closed
sectorial form. These results include the case where $\mul T$ is
invariant under $B$. The more general description makes it possible
to give an expression for the extremal maximal sectorial extensions
 of the sum of sectorial relations. In particular, one can characterize
when the form sum extension is extremal.
\end{abstract}

\maketitle

\section{Introduction}

A  linear relation $H$ in a Hilbert space $\sH$ is said to be
\textit{accretive} if $\RE (h',h) \geq 0$, $\{h,h'\} \in H$.
Note that the closure of an accretive relation is also accretive.
An accretive relation $H$ in $\sH$ is said to be
\textit{maximal accretive} if the existence of an accretive relation
$H'$ in $\sH$ with $H \subset H'$ implies $H' = H$.
A maximal accretive relation is automatically closed.
In a similar way, a  linear relation $H$ in a Hilbert space $\sH$
is said to be \textit{sectorial} with
vertex at the origin and semi-angle $\alpha$,
$\alpha \in [0,\pi /2)$, if
\begin{equation}\label{sect0-}
 | \IM (h',h) | \le (\tan \alpha) \,\RE (h',h),
 \quad  \{ h, h'\} \in H.
\end{equation}
The closure of a sectorial relation is also sectorial.
A sectorial relation $H$ in a Hilbert space $\sH$ is said to be
\textit{maximal sectorial} if the existence of a sectorial relation
$H'$ in $\sH$ with $H \subset H'$ implies $H' = H$.
A maximal sectorial relation is automatically closed.
Note that a sectorial relation is maximal sectorial if and only if
it is maximal as an accretive relation; see \cite{HSnSz09}.

A sesquilinear form $\st=\st[\cdot, \cdot]$ in a Hilbert space $\sH$
is a mapping from $\dom \st \subset \sH$  to $\dC$ which is linear
in its first entry and antilinear in its second entry. The adjoint $\st^{*}$
is defined by $\st^{*}[h,k]=\overline{\st[k,h]}$, $h,k \in \dom \st$;
for the diagonal of $\st$ the notation $\st[\cdot]$ will be used.
A (sesquilinear) form is said to be sectorial with vertex at the
origin and semi-angle $\alpha$, $\alpha \in [0,\pi/2)$, if
\begin{equation}\label{sec}
 |\st_{\I} [h]|  \le  (\tan \alpha) \, \st_{\r} [h],
 \quad
 h \in \dom \st,
\end{equation}
where the real part $\st_{\rm r}$ and the imaginary part
$\st_{\I}$ are defined by
\begin{equation}\label{secc}
\st_{\rm r}=\frac{\st+\st^{*}}{2}, \quad \st_{\rm i}=\frac{\st-\st^{*}}{2i},
\quad \dom \st_{\rm r}=\dom \st_{\I}=\dom \st.
\end{equation}
A sesquilinear form will be called a form in the rest of this note.
Observe that the form $\st_{\rm r}$ is nonnegative and that the form
$\st_{\rm i}$ is symmetric, while  $\st=\st_{\rm r}+i\,\st_{\rm i}$.
 A sectorial form $\st$  is said to be \textit{closed} if
 \[
h_n \to h, \quad \st[h_n-h_m] \to 0 \quad \Rightarrow \quad
h \in \dom \st \quad \mbox{and} \quad \st[h_n-h] \to 0.
\]
A sectorial form $\st$ is closed if and only if its real part $\st_r$ is closed;
see \cite{Kato}.

The connection between maximal sectorial relations and closed
sectorial forms is given in the so-called first representation theorem;
cf. \cite{Ar96}, \cite{HSSW17},  \cite{Kato}, \cite{RB90}.

\begin{theorem}\label{s-first}
Let $\st$ be a closed sectorial form in a Hilbert space $\sH$
with vertex at the origin and semi-angle $\alpha$, $\alpha \in [0,\pi/2)$.
Then there exists a unique maximal sectorial relation $H$ in $\sH$
with vertex at the origin and semi-angle $\alpha$
in $\sH$ such that
\begin{equation}\label{glij}
\dom H \subset \dom \st,
\end{equation}
and  for every $\{ h,h' \} \in H$ and $k \in \dom \st$ one has
\begin{equation}\label{eqn:firstrepresentation}
 \st[h,k] = ( h', k ).
\end{equation}
Conversely, for every maximal sectorial relation $H$ with vertex at
the origin and semi-angle $\alpha$, $\alpha \in [0,\pi/2)$, there
exists a unique closed sectorial form $\st$ such that \eqref{glij}
and \eqref{eqn:firstrepresentation} are satisfied.
\end{theorem}

This result contains as a special case the connection between nonnegative selfadjoint relations and
closed nonnegative forms. The nonnegative selfadjoint relation $H_r$
corresponding to the real part $\st_{\r}$ of the form $\st$ is called the real part of $H$;
this notion should not to be confused with the real part introduced in \cite{HSnSz09}. \\

In the theory of sectorial operators one encounters
expressions   $T^{*}(I+iB)T$ where $T$ is a linear operator
from a Hilbert space $\sH$ to a Hilbert space $\sK$ and $B \in \mathbf{B}(\sK)$
is a selfadjoint operator.
 In the context of sectorial relations the operator $T$
may be replaced by a linear relation $T$.
A frequently used observation is that when $T$ is a closed linear relation
and the multivalued part $\mul T$ is invariant under $B$,
then the product is a maximal sectorial relation; cf. \cite{HSSW17}.
 However, in fact,  the relation
\begin{equation}\label{tbt}
T^{*}(I+iB)T
\end{equation}
is maximal sectorial for any closed linear relation $T$.
This will be shown in this note via a matrix decomposition of $B$
with respect to the orthogonal
decomposition $\sH=\cdom T^* \oplus \mul T$.
In addition the closed sectorial form
corresponding to  $T^{*}(I+iB)T$ will be determined.
The main argument consists of a reduction to the case where $T$ is an operator.
For the convenience of the reader the arguments in the operator case
are included.
Note that if $T$ is not closed, then $T^{*}(I+iB)T$ is a sectorial relation which
may have maximal sectorial extensions, such as $T^{*}(I+iB)T^{**}$ and
some of these extensions have been determined in \cite{HSS19}; cf. \cite{ST12}.

It is clear that the sum of two sectorial relations is a sectorial relation
and there will be maximal sectorial extensions.
In \cite{HSS19} the Friedrichs extension has been determined
in general, while the Kre\u{\i}n extension has been determined
only under additional conditions. As an application of the above results
for the relation in \eqref{tbt} the Kre\u{\i}n extension and, in fact,
all extremal maximal sectorial extensions
of the sum of  two sectorial relations will be characterized in general.
With this characterization one can determine when the form sum extension is extremal.

\section{A preliminary result}

The first case to be considered is the  linear relation $T^{*}(I+iB)T$, where
$T$ a closed linear operator, which is not necessarily densely defined,
and $B \in \mathbf{B}(\sK)$ is selfadjoint.
In this case one can write down a natural closed sectorial form and
verify that $T^{*}(I+iB)T$ is the maximal sectorial relation corresponding
to the form via Theorem \ref{s-first}.

\begin{theorem}\label{s-repr0o}
Let $T$ be a closed linear operator from a Hilbert space $\sH$ to a
Hilbert space $\sK$
 and let the  operator $B \in \mathbf{B}(\sK)$ be selfadjoint.
Then the form $\st$ in $\sH$ defined by
\begin{equation}\label{henil}
 \st[h,k]=((I+iB)T h ,T k),
 \quad h,k \in \dom \st=\dom T,
\end{equation}
is closed and sectorial with vertex at the origin and semi-angle
$\alpha \leq \arctan \|B\|$ and the maximal sectorial relation $H$
corresponding to the form $\st$ is given by
\begin{equation}\label{henill}
H= T^*(I+iB)T,
\end{equation}
 with $\mul H=\mul T^*=(\dom T)^\perp$. A subset of $\dom \st = \dom
T$ is a core of the form $\st$ if and only if it is a core of the
operator $T$. Moreover, the nonnegative selfadjoint relation $H_r$
corresponding to the real part $(\st_H)_{\r}$ of the form $\st$ is
given by
\begin{equation}\label{henilll}
 H_\r=T^*T.
\end{equation}
\end{theorem}

\begin{proof}
It is straightforward to check that $\st$ in \eqref{henil}
is a closed sectorial form as indicated, since
\[
 \st_{r}[h,k]=(Th,Tk), \quad \st_{\I}[h,k]=(B Th,Tk).
\]
Therefore, $|\st_{i}[h]|=|(B Th,Th)| \leq \|B\| \|Th\|^{2}=\|B\| \st_{r}[h]$,
so that $\st$ is closed and sectorial with vertex at the origin and
semi-angle $\alpha \leq \arctan \|B\|$.
Moreover, since $T$ is closed, it is clear that $\st_{r}$
and hence $\st$ is closed.

Now let $\{h,h'\} \in T^*(I+iB)T$, then there exists
$\varphi \in \sK$ such that
\[
 \{h,\varphi\} \in T, \quad \{(I+iB)\varphi, h'\} \in T^*,
\]
from which it follows that
\[
 (h',h)=(\varphi, \varphi)+i(B \varphi, \varphi).
\]
Consequently, one sees that
\[
 | \IM (h'h) |=|(B\varphi, \varphi)| \leq \|B\|\,\|\varphi\|^2= \|B\| \,\RE (h',h),
\]
which implies that $T^*(I+iB)T$ is a sectorial relation
with vertex at the origin and semi-angle
$\alpha \leq \arctan \|B\|$. Furthermore, observe that the
above calculation also shows that $\mul T^*(I+iB)T=\mul T^*$.

To see that $T^*(I+iB)T$ is closed, let $\{h_n, h_n'\} \in T^*(I+iB)T$
converge to $\{h,h'\}$. Then there exist $\varphi_n \in \sK$ such that
\[
 \{h_n,\varphi_n\} \in T, \quad \{(I+iB)\varphi_n, h_n'\} \in T^*,
\]
and the identity $\RE(h_n',h_n)=\|\varphi_n\|^2$
shows that $(\varphi_n)$ is a Cauchy sequence
in $\sK$, so that $\varphi_n \to \varphi$ with $\varphi \in \sK$. Thus
\[
  \{h_n,\varphi_n\} \to \{h,\varphi\},  \quad
  \{(I+iB)\varphi_n, h_n'\} \to \{(I+iB)\varphi, h'\}.
\]
Since $T$ and $T^*$ are closed, one concludes
that $ \{h,\varphi\} \in T$ and $\{(I+iB)\varphi, h'\} \in T^*$,
which implies that $\{h,h'\} \in T^*(I+iB)T$.
Hence $T^*(I+iB)T$ is closed.

Now let $H$ be the maximal sectorial relation
corresponding to $\st$ in \eqref{henil}.
Assume that $\{h,h'\} \in H$, then for all $k \in \dom \st=\dom T$
\[
 \st[h,k]=(h',k) \quad  \mbox{or} \quad ( (I+iB)Th,Tk)= (h',k),
\]
which implies that
\[
 \{(I+iB)Th, h'\} \in T^* \quad \mbox{or} \quad \{h,h'\} \in T^*(I+iB)T.
\]
Consequently, it follows that $H \subset T^*(I+iB)T$.
Since $T^*(I+iB)T$ is sectorial and
$H$ is maximal sectorial, it follows that $H=T^*(I+iB)T$.
In particular, one sees that the closed relation $T^*(I+iB)T$
is maximal sectorial.
\end{proof}

With the closed linear operator $T$ from $\sH$ to $\sK$ and the selfadjoint
operator $B \in \bB(\sK)$,  consider the following matrix decomposition of $B$
\begin{equation}\label{Bdec1+}
 B=\begin{pmatrix} B_{aa}& B_{ab}\\ B_{ba}^*& B_{bb}\end{pmatrix}: \,
 \begin{pmatrix} \ker T^* \\ \cran T \end{pmatrix} \to
 \begin{pmatrix} \ker T^* \\ \cran T \end{pmatrix}.
\end{equation}
Then it is clear that
\begin{equation}\label{Bbb}
\st[h,k]=((I+iB)T h ,T k)=((I+iB_{bb})T h ,T k), \quad h,h \in \dom \st=\dom T,
\end{equation}
which shows that only the compression of $B$ to $\cran T$ plays a role in \eqref{henil}.
In applications involving Theorem \ref{s-repr0o}, it is therefore useful to recall
the following corollary.

\begin{corollary}\label{newnew}
Let $T'$ be a closed linear operator from the Hilbert space $\sH$ to a
Hilbert space $\sK'$ and let the operator $B' \in \mathbf{B}(\sK')$ be selfadjoint.
Assume that the form $\st$ in Theorem \ref{s-repr0o}  is also given by
\[
 \st[h,k]=((I+iB')T' h ,T' k),
 \quad h,k \in \dom \st=\dom T'.
\]
Then there is a unitary mapping $U$ from $\cran T$ onto $\cran T'$, such that
\[
 T'=UT, \quad B_{bb}'=UB_{bb}U^*,
\]
where $B_{bb}$ and $B_{bb}'$ stand for the compressions of $B$ and $B'$
to $\cran T$ and $\cran T'$, respectively.
\end{corollary}

\begin{proof}
By assumption $((I+iB')T' h ,T' k)= ((I+iB)T h ,T k)$ for all $h,k \in \dom \st$. This leads to
\[
 (T' h ,T' k)= (T h ,T k)  \quad \mbox{and} \quad (B' T' h ,T' k)= (B T h ,T k)
\]
for all $h,k \in \dom \st$. Hence the mapping $Th \mapsto T'h$ is unitary, and denote it by $U$.
Then $T'=UT$ and it follows that $(B' T' h ,T' k)= (B U^*T' h , U^*T' k)$.
\end{proof}

\section{A matrix decomposition for $T^*(I+iB)T$}\label{sect3}

Let $T$ be a linear relation from $\sH$ to $\sK$ which is closed;
observe that then the subspace $\mul T$ is closed.
The adjoint $T^{*}$ of $T$ is the set of all $\{h,h'\} \in \sK \times \sH$
for which
\[
 (h',f)=(h,f') \quad \mbox{for all} \quad \{f,f'\} \in T.
\]
Hence, the definition of $T^{*}$ depends on the Hilbert spaces
$\sH$ and $\sK$ in which $T$ is assumed to act.
Let $\sK$ have the orthogonal decomposition
\begin{equation}\label{hilbert}
\sK=\cdom T^* \oplus \mul T,
\end{equation}
and let $P$ be the orthogonal projection onto $\cdom T^*$.
Observe that $PT \subset T$,
since $\{0\} \times \mul T \subset T$.
Therefore $T^{*} \subset (PT)^{*}=T^*P$,
where the last equality holds since $P \in \mathbf{B}(\sK)$.
Then one has
\begin{equation}\label{hilbert1}
 (PT)^*=T^* \hoplus\, (\mul T \times \{0\}).
\end{equation}
The \textit{orthogonal operator part} $T_{\rm s}$
of $T$ is defined as $T_{\rm s}=PT$.
Hence $T_\op$ is an operator from
the Hilbert space $\sH$ to the
Hilbert space $\sK$ and $T_\op \subset T$.
Note that
$\ran T_\op \subset \cdom T^*=\sK \ominus \mul T$.
Thus one may interpret $T_\op$ as an operator
from the Hilbert space $\sH$
to the Hilbert space $\cdom T^*$
and one may also consider the adjoint $(T_\op)^\times$
of $T_{\s}$ with respect to these spaces. It is not difficult to see
the connection between these adjoints:  if $\{h,h'\} \in \sK \times \sH$, then
\begin{equation}\label{hilbert17}
 \{h,h'\} \in T^{*} \quad \Leftrightarrow \quad \{h,h'\} \in (T_{\s})^{\times}.
\end{equation}
The identity \eqref{hilbert1} shows the difference between
$(T_{\s})^{*}$ and $(T_\op)^\times$.

Let $T$ be a closed linear relation from a Hilbert space $\sH$ to a
Hilbert space $\sK$ and let $B \in \mathbf{B}(\sK)$ be selfadjoint.
In order to study the linear relation
\[
 T^{*}(I+iB)T,
\]
decompose the Hilbert space $\sK$ as in \eqref{hilbert}
and decompose the selfadjoint operator $B \in \mathbf{B}(\sK)$
accordingly:
\begin{equation}\label{Bdec1}
 B=\begin{pmatrix} B_{11}& B_{12}\\ B_{12}^*& B_{22}\end{pmatrix}: \,
 \begin{pmatrix} \cdom T^* \\ \mul T \end{pmatrix} \to
 \begin{pmatrix} \cdom T^* \\ \mul T \end{pmatrix}.
\end{equation}
Here the operators $B_{11} \in \mathbf{B}(\cdom T^*)$
and $B_{22} \in \mathbf{B}(\mul T)$ are selfadjoint,
while $B_{12} \in \mathbf{B}(\mul T, \cdom T^*)$
and $B_{12}^* \in \mathbf{B}(\cdom T^*,\mul T)$.

By means of the decomposition \eqref{Bdec1} the following
auxiliary operators will be introduced.
First, define the operator $C_0 \in \mathbf{B}(\cdom T^*)$ by
\begin{equation}\label{c0}
C_0=I+B_{12}(I+B_{22}^2)^{-1}B_{12}^*.
\end{equation}
Observe that $C_0 \geq I$ and that $(C_0)^{-1}$
belongs to $\mathbf{B}(\cdom T^*)$
and is a nonnegative operator.
Next, define the operator $C \in \mathbf{B}(\cdom T^*)$  by
\begin{equation}\label{c}
 C=C_0^{-\half}\left[B_{11}-B_{12}(I+B_{22}^2)^{-\half}B_{22}
 (I+B_{22}^2)^{-\half}B_{12}^*\right]C_0^{-\half},
\end{equation}
which is clearly selfadjoint.

\begin{lemma}\label{redux}
Let $T$ be a closed linear relation from a Hilbert space $\sH$ to a
Hilbert space $\sK$,
 let $T_\op$ be the orthogonal operator part of $T$, and
let the selfadjoint operator $B \in \mathbf{B}(\sK)$ be decomposed
as in \eqref{Bdec1}.
Let the operators $C_0$ and $C$ be defined by
\eqref{c0} and \eqref{c}. Then
\begin{equation}\label{redu}
 T^*(I + i B)T = (T_{\s})^{\times}C_0^{1/2}(I+i C) C_0^{1/2}T_{\s},
\end{equation}
and, consequently, $T^*(I + i B)T$ is maximal sectorial and
\begin{equation}\label{redu1}
 \mul T^*(I + i B)T=\mul T^{*}=\mul (T_{\s})^{\times}.
\end{equation}
\end{lemma}

\begin{proof}
 In order to prove the equality
in \eqref{redu}, assume that $\{h,h'\}\in T^* (I + i B) T$. This means that
\begin{equation}\label{eka-}
\{h,\varphi\}\in T \quad \mbox{and} \quad \{(I + i B)\varphi,h'\}\in T^*
\end{equation}
for some $\varphi\in\sK$. Decompose the element $\varphi$ as
\begin{equation}\label{eka--}
\varphi=\varphi_1+\varphi_2, \quad \varphi_{1} \in \cdom T^{*},
\,\,\varphi_2\in\mul T.
\end{equation}
Since $\{0, \varphi_{2}\} \in T$, it is clear that
\begin{equation}\label{eka---}
\{h,\varphi\}\in T \quad \Leftrightarrow \quad  \{h,\varphi_{1}\}\in T_{\s}.
\end{equation}
Using \eqref{eka--} and the above decomposition \eqref{Bdec1} of $B$,
one observes that
\[
 \{(I + i B)\varphi,h'\}
 =\left\{ \begin{pmatrix}   (I+i B_{11})\varphi_1+iB_{12}\varphi_2 \\
                                          iB_{12}^*\varphi_1+(I+i B_{22})\varphi_2
              \end{pmatrix},
   h' \right\},
\]
which implies that the condition $\{(I + i B)\varphi,h'\}\in T^*$ is equivalent to
 \[
 \left\{ \begin{array}{l} \{ (I+i B_{11})\varphi_1+iB_{12}\varphi_2, h' \} \in T^*,
           \\ iB_{12}^*\varphi_1+(I+i B_{22})\varphi_2=0,
           \end{array}
           \right.
\]
or, what is the same thing,
 \begin{equation}\label{eka}
 \left\{ \begin{array}{l} \{ [I+i B_{11} + B_{12}(I+i
  B_{22})^{-1}B_{12}^*]\varphi_1, h'\} \in T^*,  \\
  \varphi_2=-i(I+i B_{22})^{-1}B_{12}^*\varphi_1.
           \end{array}
          \right.
\end{equation}
Due to the definitions \eqref{c0} and \eqref{c} and the identity
\[
(I+i B_{22})^{-1} =
(I+B_{22}^2)^{-\half}(I-iB_{22})(I+B_{22}^2)^{-\half},
\]
observe that
\[
\begin{split}
&I+i B_{11} + B_{12}(I+i B_{22})^{-1}B_{12}^*  \\
&\hspace{1cm}  =C_0 +
 i[B_{11}-B_{12}(I+B_{22}^2)^{-\half}B_{22}(I+B_{22}^2)^{-\half}B_{12}^*]\\
&\hspace{1cm}  = C_0^{1/2}(I+i C) C_0^{1/2}.
\end{split}
\]
Therefore, it follows from \eqref{eka}, via the equivalence in \eqref{hilbert17}, that
\begin{equation}\label{eka+}
\{(I + i B)\varphi, h'\} \in  T^* \quad \Leftrightarrow \quad
 \left\{ \begin{array}{l} \{C_0^{1/2}(I+i C) C_0^{1/2} \varphi_1,h'\} \in (T_{s})^{\times}, \\
  \varphi_2=-i(I+i B_{22})^{-1}B_{12}^*\varphi_1.
           \end{array}
           \right.
\end{equation}
 Combining  \eqref{eka---} and \eqref{eka+}, one sees that
\[
 \{h,h'\} \in (T_{\s})^{\times} C_0^{1/2}(I+i C) C_0^{1/2} T_{\s}.
\]
Conversely, if this inclusion holds, then there exists $\varphi_{1} \in \cdom T^{*}$,
such that
\[
 \{h, \varphi_{1}\} \in T_{\s} \quad \mbox{and}
 \quad \{C_0^{1/2}(I+i C) C_0^{1/2} \varphi_1,h'\} \in (T_{\s})^{\times}.
\]
Then define $\varphi_2=-i(I+i B_{22})^{-1}B_{12}^*\varphi_1$,
so that $\varphi_{2} \in \mul T$. Furthermore, define $\varphi=\varphi_{1}+\varphi_{2}$.
Hence $\{h, \varphi\} \in T$, and it follows from \eqref{eka+}
that
\[
\{h,h'\}\in T^* (I + i B) T.
\]
Therefore one can rewrite $T^* (I + i B) T$ in the form \eqref{redu}.

Observe that $C_{0}^{\half}T_{\s}$ is a closed linear operator
from the Hilbert space $\sH$ to the Hilbert space $\cdom T^{*}$
whose adjoint is given by
\begin{equation}\label{C0Tadjoint}
 ( C_0^{1/2}T_{\rm s})^{\times}=(T_{\rm s})^{\times} \, C_0^{1/2}.
\end{equation}
Hence, by Theorem \ref{s-repr0o}
$(T_{\s})^{\times} C_0^{1/2}(I+i C) C_0^{1/2} T_{\s}$ is a maximal sectorial
relation in $\sH$ and by the identity \eqref{redu} the same is true for $T^*(I+iB)T$.

The statement in \eqref{redu1} follows by tracing the above equivalences
for an element $\{0,h'\}$.
 \end{proof}

\begin{remark}\label{rem3.2}
Let $\varphi=\varphi_1+\varphi_2\in \sK$ be decomposed as in \eqref{eka--}.
Then one has the following equivalence:
\[
 (I+iB)\varphi\in\cdom T^* \quad \Leftrightarrow\quad (I+iB)\varphi=C_0^{1/2}(I+iC)C_0^{1/2}\varphi_1.
\]
To see this, let $\eta =(I+iB)\varphi$. Then $\eta \in \cdom T^*$ if and only if
\[
 \begin{pmatrix}   I+i B_{11} & iB_{12} \\
                                          iB_{12}^* & I+i B_{22}
              \end{pmatrix}
  \begin{pmatrix} \varphi_1 \\ \varphi_2 \end{pmatrix}
  =   \begin{pmatrix} \eta \\ 0 \end{pmatrix},
\]
where $\cdom T^*$ is interpreted as the subspace $\cdom T^*\times \{0\}$ of $\sK$.
Now apply \eqref{eka+}.
\end{remark}

\section{A class of maximal sectorial relations and associated forms}\label{sect4}

The linear relation $T^*(I+iB)T$ is maximal sectorial
for any selfadjoint $B \in \mathbf{B}(\sK)$
and any closed linear relation $T$ from $\sH$ to $\sK$.
Now the corresponding closed sectorial form will be determined.
This gives the appropriate version of Theorem \ref{s-repr0o} in terms of relations.
In fact, the general result is based on a reduction via
Lemma \ref{redux} to Theorem \ref{s-repr0o}.

\begin{theorem}\label{s-repr}
Let $T$ be a closed linear relation from a Hilbert space $\sH$ to a
Hilbert space $\sK$ and
 let the selfadjoint operator $B \in \mathbf{B}(\sK)$ be decomposed
as in \eqref{Bdec1}.
Let the operators $C_0$ and $C$ be defined by \eqref{c0} and \eqref{c}.
 Then the form $\st$ defined by
\begin{equation}\label{henilB}
\st[h,k]=((I+iC) \,C_0^{\half}T_{\s} \,h , C_0^{\half}T_{\s}\, k),
 \quad h,k \in \dom \st= \dom T,
\end{equation}
is closed and sectorial with vertex at the origin and semi-angle
$\gamma \leq \arctan \|C\|$. Moreover,
the maximal sectorial relation $H$ corresponding to the form
$\st$ is given by
\begin{equation}\label{HwithC}
 H = (T_{\s})^{\times}\, C_0^{1/2}(I+i C) C_0^{1/2}\,T_{\s}=T^*(I + i B)T.
\end{equation}
 A subset of $\dom \st = \dom T$ is a core of the form $\st$ if
and only if it is a core of the operator $T_\op$. Moreover, the
nonnegative selfadjoint relation $H_r$ corresponding to the real
part $(\st_H)_r$ of the form $\st$ is given by
\[
H_r=(T_{\s})^{\times}C_0T_{\s}.
\]
\end{theorem}

\begin{proof}
Since $C_{0}^{\half}T_{\s}$ is a closed linear operator
from the Hilbert space
$\sH$ to the Hilbert space $\cdom T^{*}$,
Theorem \ref{s-repr0o} (with $\sK$ replaced by $\cdom T^{*}$,
$B$ by $C$, and $T$ by $C_0^{1/2}T_{\rm s}$) shows that the form $\st$
in \eqref{henilB} is closed and sectorial with vertex at the origin and semi-angle
$\gamma \leq \arctan \|C\|$.
Moreover, the maximal sectorial relation associated with the form
$\st$ is given by
\[
 ( C_0^{1/2}T_{\rm s})^{\times}(I+i C) C_0^{1/2}T_{\rm s}
 = (T_{\rm s})^{\times}C_0^{1/2}(I+i C) C_0^{1/2}T_{\rm s},
\]
cf. \eqref{henil}, \eqref{henill}, and \eqref{C0Tadjoint}.
The identities in \eqref{HwithC} are clear from Lemma \ref{redux}.
The assertion concerning the core holds, since
the factor $C_0$ is bounded with bounded inverse. The formula
\eqref{HwithC} shows that
\[
 (\st_H)_r[h,k]=(C_0^{\half}T_\op h ,C_0^{\half}T_\op k),
 \quad h,k \in \dom \st = \dom T,
\]
and hence
$H_{\rm r}=(C_0^{1/2}T_\op)^{\times}C_0^{1/2}T_\op
=(T_\op)^{\times}C_0 T_\op$
(cf.\,\,the discussion above).
\end{proof}

Recall  that if $\{h,h'\}\in T^* (I + i B) T$,
then $\{h,\varphi\}\in T$ and $\{(I + i B)\varphi,h'\}\in T^*$.
The  last inclusion implies the condition
$(I+iB)\varphi \in \dom T^{*}\subset \cdom T^{*}$,
giving rise to $\varphi_2=-i(I+i B_{22})^{-1}B_{12}^*\varphi_1$.
Thus, for instance, when $B=\diag (B_{11}, B_{22})$,
it follows that $\varphi_{2}=0$, so that it is immediately clear that
$\gamma \leq \arctan \|B_{11}\|$, independent of $B_{22}$.
 Note that the following assertions are equivalent:
\begin{enumerate}\def\labelenumi{\rm (\roman{enumi})}
\item  $B=\diag (B_{11}, B_{22})$;
\item  $B_{12}=0$;
\item  $C_0=I$;
\item  $\mul T$ is invariant under $B$,
\end{enumerate}
in which case $C=B_{11}$. Hence,
if $\mul T$ is invariant under $B$, i.e.,
if any of the assertions (i)--(iv) hold, then
Theorem \ref{s-repr} gives the following corollary,
which coincides with \cite[Theorem 5.1]{HSSW17}.
In the case where $\mul T=\{0\}$ the corollary
 reduces to Theorem \ref{s-repr0o}.

\begin{corollary}\label{old}
Let $T$ be a closed linear relation from a Hilbert space $\sH$ to a
Hilbert space $\sK$,
 let $T_\op$ be the orthogonal operator part of $T$, and
let $\mul T$ be invariant under the selfadjoint operator $B \in \mathbf{B}(\sK)$,
so that $B=\diag (B_{11}, B_{22})$.
Then the form $\st$ defined by
\[
\st[h,k]=((I+iB_{11})T_{\s}h, T_{\s}k), \quad h,k \in \dom \st=\dom T,
\]
is closed and sectorial with vertex at the origin and semi-angle
$\gamma \leq \arctan \|B_{11}\|$. Moreover,
the maximal sectorial relation $H$ corresponding to the form
$\st$ is given by
\[
 H=(T_{\s})^{\times} (I+iB_{11}) T_{\s}=T^{*}(I+iB)T.
\]
\end{corollary}

In the case that $\mul T$ is not invariant under $B$, one has  $C_0\neq I$,
and the formulas are different: for instance,
the real part $(\st_H)_r$ in Theorem \ref{s-repr}
is of the form
\[
(\st_H)_r[h,k]=(C_0^{\half}T_\op h ,C_0^{\half}T_\op k),
 \quad h,k \in \dom \st=\dom T_\op = \dom T.
\]

\begin{example}\label{examp}
Assume that $B_{11} \neq 0$ and
\[
B_{11}=B_{12}(I+B_{22}^2)^{-\half}B_{22}(I+B_{22}^2)^{-\half}B_{12}^*,
\]
so that $C=0$. In this case the maximal sectorial relation $H=T^*(I+iB)T$
in Theorem \ref{s-repr} is selfadjoint, i.e., $H=H_r$ and the associated form
$\st$ is nonnegative.
On the other hand, with such a choice of $B$ the operator part of $T$
determines the maximal sectorial relation
$(T_{\rm s})^*(I + iB)T_{\rm s}$ with semi-angle $\arctan \|B_{11}\|>0$,
while $T^*(I + i B)T$ has semi-angle $\gamma=0$.
\end{example}

\section{Maximal sectorial relations and their representations}

Let $H$ be a maximal sectorial relation in $\sH$ and let the closed
sectorial form $\st_H$ correspond to $H$; cf. Theorem \ref{s-first}.
Since the closed form $\st_H$ is sectorial, one has the inequality
\begin{equation}\label{henillla0}
 |(\st_{H})_{\rm i}[h] | \leq (\tan \alpha) (\st_{H})_{\rm r}[h], \quad h \in \dom \st,
\end{equation}
and in this situation the real part  $(\st_{H})_{\rm r}$ is a closed nonnegative form.
Hence by the first representation theorem there exists a nonnegative selfadjoint
relation $H_{\rm r}$, the so-called \textit{real part} of $H$, such that
$\dom H_{\rm r} \subset \dom (\st_{H})_{\rm r}=\dom \st_{H}$ and
\[
(\st_{H})_{\rm r}[h,k] = ( h', k ), \quad \{ h,h' \} \in H_{\rm r},
\quad k \in \dom (\st_{H})_{\rm r}=\dom \st_{H}.
\]
This real part $H_{\rm r}$, not to be confused with the real part
introduced in \cite{HSnSz09}, will play an important role in
 formulating the second representation theorem below.
First the case where $H$ is a maximal sectorial operator will be considered,
in which case $H$ is automatically densely defined; see \cite{Kato}.

\begin{lemma}\label{s-thirdLemma}
Let $H$ be an maximal sectorial operator in $\sH$, let the closed
sectorial form $\st_H$ correspond to $H$ via Theorem \ref{s-first},
and let $H_{\rm r}$ be the real part of $H$.
Then there exists a unique selfadjoint operator $G \in \mathbf{B}(\sH)$
with $\| G \| = \tan \alpha$, of the form
\begin{equation}\label{Bdec1++}
 G=\begin{pmatrix} 0& 0\\ 0& G_{bb}\end{pmatrix}: \,
 \begin{pmatrix} \ker H_{\rm r} \\ \cran H_{\rm r} \end{pmatrix} \to
 \begin{pmatrix} \ker H_{\rm r} \\ \cran H_{\rm r} \end{pmatrix},
\end{equation}
such that
 \begin{equation}\label{sss--}
 \st_H[h,k]=((I +iG) (H_{r})^{\half} h, (H_{r})^{\half} k),
 \quad
 h,k \in \dom \st_H =\dom H_{r}^{\half}.
\end{equation}
Moreover, the corresponding maximal sectorial operator $H$ is given by
\[
 H=(H_{\rm r})^{\half}(I+iG)(H_{\rm r})^{\half},
 \]
with $\mul H=\mul  H_{\rm r}$.
\end{lemma}

\begin{proof}
The inequality
\[
 | (\st_{H})_{\rm i}[h,k] |^{2} \leq C \st_{\rm r}[h] \st_{\rm r}[k]
 =C\|H_{\rm r}^{\half} h\| \|H_{\rm r}^{\half} k\|, \quad h,k \in \dom,
\]
shows the existence of a selfadjoint operator $G$ in
$\sH \ominus \ker H$ such that
\begin{equation}\label{henilllaa}
 (\st_{H})_{\rm i}[h,k]=(G (H_{\rm r})^{\half} h, (H_{\rm r})^{\half} k),
 \quad h, k \in \dom (H_{\rm r})^{\half}.
\end{equation}
Extend $G$ to all of $\sH$ in a trivial way, so that the same formula remains valid;
see Corollary \ref{newnew}.
It follows from the decomposition $\st=\st_{\r} +i \st_{\I}$, cf. \eqref{secc},
and the identities \eqref{henillla} and  \eqref{henilllaa},  that
\[
\st_{H}= (\st_{H})_{\rm r}+i (\st_{H})_{\rm i},
\]
so that
\[
\st_{H}=[h,k]= ( (H_{r})_{\s}^{\half}h, (H_{r})_{\s}^{\half} k)
+i ( G(H_{r})_{\s}^{\half}h, (H_{r})_{\s}^{\half} k).
\]
This last identity immediately gives  \eqref{sss--}.
The rest follows from Corollary \ref{old}.
\end{proof}

Now let $H$ be a maximal sectorial relation, let $H_{\r}$ be its real part,
and let $(H_{r})_{\s}$ be its orthogonal operator part.
 Then one obtains the representation
\begin{equation}\label{henillla}
 (\st_{H})_{\rm r}[h,k]=( ((H_{r}))_{\s}^{\half}h, ((H_{r})_{\s})^{\half} k), \quad
 h,k \in \dom (\st_{H})_{\rm r}=\dom ((H_{r})_{\s})^{\half},
\end{equation}
cf. Theorem \ref{s-repr0o}. Now apply  Corollary \ref{old}
and  therefore one may formulate the second representation theorem as follows.

\begin{theorem}\label{s-third}
Let $H$ be a maximal sectorial relation in $\sH$, let the closed
sectorial form $\st_H$ correspond to $H$ via Theorem \ref{s-first},
and let $H_{\rm r}$ be the real part of $H$.
Then there  exists a selfadjoint operator $G \in \mathbf{B}(\sH)$
with $\| G \| = \tan \alpha$, such that $G$ is trivial on
$\ker H_{\rm r} \oplus \mul H_{\rm r}$, and
 \begin{equation}\label{sss-}
 \st_H[h,k]=((I +iG) ((H_{\rm r})_{\rm s})^{\half} h, ((H_{\rm r})_{\rm s})^{\half} k),
 \quad
 h,k \in \dom \st_H =\dom H_{\rm r}^{\half}.
\end{equation}
Moreover, the maximal sectorial relation $H$ is given by
\begin{equation}\label{sss-+}
 H=(((H_{\rm r})_{s})^{\half})^{\times}(I+iG)((H_{\rm r})_{s})^{\half},
 \end{equation}
with $\mul H=\mul H_{\rm r}$.
\end{theorem}

Next, it is assumed that $H$ is a maximal sectorial relation of the form $H=T^{*}(I+iB)T$,
where $T$ is a closed linear relation from a Hilbert space $\sH$ to a
Hilbert space $\sK$ and the operator $B \in \mathbf{B}(\sK)$ is selfadjoint.
Let the operators $C_0$ and $C$ be defined by \eqref{c0} and \eqref{c}, then
 \[
 H= (T_{\s})^{\times}\, C_0^{1/2}(I+i C) C_0^{1/2}\,T_{\s},
\]
while the corresponding closed sectorial form is given
\[
\st[h,k]=((I+iC) \,C_0^{\half}T_{\s} \,h , C_0^{\half}T_{\s}\, k),
 \quad h,k \in \dom \st= \dom T.
\]
To compare these expressions with \eqref{sss-} and \eqref{sss-+},
observe that
\[
 (\,C_0^{\half}T_{\s} \,h , C_0^{\half}T_{\s}\, k)
 =((H_{r})_{\rm s})^{\half} h, ((H_{r})_{\rm s})^{\half} k)
\]
and
\[
 (C \,C_0^{\half}T_{\s} \,h , C_0^{\half}T_{\s}\, k)
 =(G ((H_{r})_{\rm s})^{\half} h, ((H_{r})_{\rm s})^{\half} k).
\]
It is clear from \eqref{henilB} that only the (selfadjoint) compression
of $C$ to $\cran C_0^{1/2}T_{\rm s}$
contributes to the form \eqref{HwithC},
so that it is straightforward to set up a unitary mapping; cf. Corollary \ref{newnew}.

\section{Extremal maximal sectorial extensions of sums of maximal sectorial relations}

Let $H_{1}$ and $H_{2}$ be maximal sectorial relations in a Hilbert space $\sH$.
Then the sum $H_{1}+H_{2}$ is a sectorial relation  in $\sH$ with
\[
\dom (H_{1}+H_{2})=\dom H_{1} \cap \dom H_{2},
\]
so that the sum is not necessarily densely defined. In particular,
$H_1+H_2$ and its closure need not be operators, since
\begin{equation}
\label{s-sum-mulll}
 \mul (H_{1}+H_{2})=\mul H_{1} + \mul H_{2}.
\end{equation}
To describe the class of extremal maximal sectorial extensions of $H_1+H_2$
some basic notations are recalled from \cite{HSS19}, together with
the description of the Friedrichs and  Kre\u{\i}n extensions
\[
(H_1+H_2)_F \quad \mbox{and} \quad (H_1+H_2)_K
\]
of $H_1+H_2$, respectively.
In order to describe the whole class of extremal extensions of $H_1+H_2$
and the corresponding closed forms a proper description
of the closed sectorial form $\st_K$ is essential.
The results in Sections \ref{sect3} and \ref{sect4} allow a general treatment
that will relax the additional conditions   in \cite{HSS19}.

\subsection{Basic notions}
Let $H_{1}$ and $H_{2}$ be maximal sectorial relations and
decompose them as follows
\begin{equation}\label{H12}
 H_{j} = A_{j}^{\half} (I+iB_{j}) A_{j}^{\half},
\quad 1 \le j \le 2,
\end{equation}
where $A_{j}$ (the real part of $H_{j}$), $1 \le j \le 2,$ are
nonnegative selfadjoint relations in $\sH$ and $B_{j}$, $1 \le j \le
2,$ are  bounded selfadjoint operators in $\sH$ which are
trivial on $\ker A_j \oplus \mul A_j$; cf. Theorem \ref{s-third}.
Furthermore, if $A_{1}$ and $A_{2}$ are decomposed as
\[
  A_{j}=A_{j\s} \oplus A_{j \infty}, \quad 1 \le j \le 2,
\]
where $A_{j \infty} =\{0\} \times \mul A_{j}$,
$1 \le j \le 2$, and $A_{j\s}$, $1 \le j \le 2$, are densely
defined nonnegative selfadjoint operators
(defined as orthogonal complements in the graph sense),
then the uniquely determined square roots of $A_{j}$, $1 \le j \le 2$
are given by
\[
   A_{j}^\half=A_{j\s}^\half \oplus A_{j \infty},
   \quad 1 \le j \le 2.
\]
Associated with $H_{1}$ and $H_{2}$ is the relation
$\Phi$ from $\sH \times \sH$ to $\sH $, defined by
\begin{equation}
\label{s-s-Einz}
 \Phi= \left\{\, \left\{ \{f_{1},f_{2}\}, f_{1}' + f_{2}' \right\}
 :\, \{f_{j},f_{j}' \} \in A_{j}^{\half} , \,
 1 \le j \le 2\,\right\}.
\end{equation}
Clearly, $\Phi$ is a relation whose domain and multivalued part
are given by
\[
 \dom \Phi = \dom A_{1}^{\half} \times \dom A_{2}^{\half},
 \quad \mul \Phi=\mul H_{1} +\mul H_{2}.
\]
The relation $\Phi$ is not necessarily densely defined in $\sH
\times \sH$, so that in general $\Phi^*$ is a relation  as $\mul
\Phi^*=(\dom \Phi)^\perp$.
Furthermore, the adjoint $\Phi^*$ of $\Phi$
is the relation from $\sH$ to $\sH \times \sH$, given by
\begin{equation}
\label{s-s-Zwei}
\Phi^* = \left\{ \left\{ h, \{ h_{1}', h_{2}'\} \right\} \, :
\{h, h_{j}'\} \in A_{j}^{\half}, \, 1 \le j \le 2 \right\}.
\end{equation}
The identity \eqref{s-s-Zwei} shows that the (orthogonal)
operator part
$(\Phi^*)_\s$ of $\Phi^*$ is given by:
\begin{eqnarray}
\label{s-s-qus}
 (\Phi^*)_\s & = &\left\{ \left\{ h, \{ h_{1}', h_{2}'\} \right\} \,
 : \{h, h_{j}'\} \in A_{j\s}^{\half}, \, 1 \le j \le 2  \right\}
 \\
& = & \left\{ \left\{ h, \{ A_{1\s}^{\half} h, A_{2\s}^{\half} h\} \right\}
\, : h \in \dom A_{1}^{\half} \cap \dom A_{2}^{\half} \right\}.
\nonumber
\end{eqnarray}
The identities \eqref{s-s-Zwei} and \eqref{s-s-qus} show that
\[
\dom  \Phi^{*}  = \dom A_{1}^{\half} \cap \dom A_{2}^{\half},
\, \, \mul \Phi^*=\mul H_{1} \times \mul H_{2},
\, \, \ran (\Phi^{*})_\s  = \sF_{0},
\]
where the subspace $\sF_{0} \subset \sH \times \sH$ is defined by
\begin{equation}\label{F0}
\sF_{0} = \left\{ \left\{ A_{1\s}^{\half} h , A_{2\s}^{\half} h
\right\} \, : \, h \in \dom A_{1}^{\half} \cap \dom A_{2}^{\half}
\right\}.
\end{equation}
The closure of $\sF_{0}$ in $\sH \times \sH$ will be denoted by
$\sF$. Define the relation $\Psi$ from $\sH$ to $\sH \times \sH$
by
\begin{equation}
\label{s-s-qu}
 \Psi=\left\{\, \left\{h, \left\{A_{1s}^{\half} h, A_{2s}^{\half}
 h\right\}\right\} :\, h \in \dom H_{1} \cap \dom H_{2} \,\right\} \subset
 \sH \times (\sH \times \sH).
\end{equation}
It follows from this definition that
\[
 \dom \Psi=\dom H_{1} \cap \dom H_{2}, \quad
  \mul \Psi=\{0\}, \quad
 \ran \Psi=\sE_0,
\]
where the space $\sE_0 \subset \sH \times \sH$ is defined by
\begin{equation}\label{E0}
 \sE_0  =  \left\{\, \left\{A_{1\s}^{\half}f , A_{2\s}^{\half}f\right\}:\,
        f\in \dom H_{1} \cap \dom H_{2} \,\right\}.
\end{equation}
Observe that $\sE_0 \subset \sF_0$.
The closure of $\sE_{0}$ in $\sH \times \sH$ will be denoted by
$\sE$. Hence,
\begin{equation}
\label{s-s-ef}
          \sE \subset \sF.
\end{equation}
Comparison of \eqref{s-s-qus} and \eqref{s-s-qu} shows
\begin{equation}\label{Psiclos}
\Psi \subset (\Phi^*)_\s,
\end{equation}
and thus the operator $\Psi$ is closable and $\Psi^{**} \subset (\Phi^*)_\s$.
 It follows from
$\cdom \Psi^*=(\mul \Psi^{**})^\perp$ and
$\mul \Psi^*=(\dom \Psi)^\perp$, that
\[
  \cdom \Psi^*=\sH , \quad \mul \Psi^* =(\dom H_{1} \cap \dom H_{2})^\perp.
\]
Next, define the relation $K $ from $\sH \times \sH$ to $\sH$ by
\begin{eqnarray}\label{s-s-opK}
K & = &
 \big\{ \{\{
(I+iB_{1})A_{1\s}^{\half} f, (I+iB_{2})A_{2\s}^{\half}f \}, f_{1}' + f_{2}' \}\, :
 \\
&&
\quad \quad \quad \{ (I+iB_{1})A_{1\s}^{\half} f, f_{1}' \} \in A_{1}^{\half},
\{ (I+iB_{2})A_{2\s}^{\half} f, f_{2}' \} \in A_{2}^{\half} \big\}
\nonumber \\
&&
\hspace{-0.6cm} \subset \hspace{0.2cm} (\sH \times \sH) \times \sH. \nonumber
\end{eqnarray}
Clearly, the domain and multivalued part of $K$ are
given by
\[
 \dom K= \sD_0, \quad \mul K=\mul (H_{1} + H_{2}),
\]
where
\begin{equation}\label{domK}
   \sD_{0}= \left\{\,\{(I+iB_{1}) A_{1\s}^{1/2}f,
   (I+iB_{2}) A_{2\s}^{1/2}f\} :
   \,  f\in \dom H_{1} \cap \dom H_{2}
  \,\right\} .
\end{equation}
The closure of $\sD_{0}$ in $\sH \times \sH$ will be denoted by
$\sD$.

\begin{lemma}\label{lem3.1}
The relations $K$, $\Phi$, and $\Psi$ satisfy the following
inclusions:
\begin{equation}
\label{s-s-trits}
  K \subset \Phi \subset \Psi^*, \quad \Psi \subset \Phi^* \subset
K^*.
\end{equation}
\end{lemma}

\begin{proof}
To see this note that $K \subset \Phi$ follows from \eqref{s-s-Einz}
and \eqref{s-s-opK}, and that $\Psi \subset \Phi^*$ follows from
\eqref{s-s-Zwei} and \eqref{s-s-qu}. Therefore, also
$\Phi^* \subset K^*$ and $\Phi \subset \Phi^{**} \subset \Psi^*$.
\end{proof}

\subsection{The Friedrichs and the Kre\u{\i}n extensions of
$H_1+H_2$}\label{sec3.2}

The descriptions of the Friedrichs extension and
the Kre\u{\i}n extension $(H_{1}+H_{2})_F$ and $(H_{1}+H_{2})_K$
of $H_1+H_2$ are now recalled from \cite{HSS19}.
For this, define the orthogonal sum of the operators
$B_{1}$ and $B_{2}$ in $\sH \times \sH$ by
\[
 B_\oplus:=B_1 \oplus B_2 =
\begin{pmatrix}
B_{1} & 0 \\
0 & B_{2}
\end{pmatrix}.
\]
The descriptions of $(H_{1}+H_{2})_F$ and $(H_{1}+H_{2})_K$
incorporate the initial data on the factorizations \eqref{H12} of $H_1$ and $H_2$
via the mappings $\Phi$, $\Psi$, and $K$ in Subsection \ref{s-sum-mulll}.
The construction of the Friedrichs extension was given in \cite[Theorem~3.2]{HSS19},
where some further details and a proof of the following result can be found.
The new additions in the next theorem are the second representations for
$(H_{1}+H_{2})_F$ and $\st_{F}$ that will be needed in the rest of this paper.

\begin{theorem} \label{ss-twee}
Let $H_{1}$ and $H_{2}$ be maximal sectorial and let $\Psi$ be defined by \eqref{s-s-qu}.
Then the Friedrichs extension of $H_{1}+H_{2}$ has the expression
\begin{equation}\label{HF2}
 (H_{1}+H_{2})_F=\Psi^* (I + i B_\oplus) \Psi^{**}=\Psi^* C_0^{1/2}(I+iC)C_0^{1/2} P_\sD (\Psi^{**})_{\s}.
\end{equation}
The closed sectorial form $\st_{F}$  associated with $(H_{1}+H_{2})_F$  is given by
\begin{equation}\label{tF2}
\st_{F} [f,g]= ((I + iB_\oplus ) \Psi^{**} f, \Psi^{**} g)
 = (C_0^{1/2}(I+iC)C_0^{1/2} P_\sD (\Psi^{**})_{\s} f,P_\sD (\Psi^{**})_{\s} g),
 \end{equation}
for all $f, \, g  \in \dom \st_{F} = \dom \Psi^{**}$.
\end{theorem}
\begin{proof}
As indicated the first expressions for $(H_{1}+H_{2})_F$ in \eqref{HF2} and $\st_F$ in \eqref{tF2}
have been proved in \cite[Theorem~3.2]{HSS19} and, hence,
it suffices to derive the second
expressions in \eqref{HF2} and \eqref{tF2}.

By definition, one has  $\ran \Psi=\sE_0$ (see \eqref{s-s-qu}, \eqref{E0}),
and by Lemma \ref{lem3.1} one has $\Psi\subset \Psi^{**}\subset K^*$, which
after projection onto $\sD=\cdom K$ yields
\[
 P_\sD \Psi^{**}\subset P_\sD K^*=(K^*)_\s.
\]
Notice that $\sD_0=\dom K=(I+i B_\oplus)\sE_0$ (see \eqref{E0}, \eqref{domK}).
Since the operator $I+iB_\oplus$ is bounded with bounded inverse,
one has the equality
\begin{equation}\label{sDsE}
  \sD=(I+i B_\oplus)\sE.
\end{equation}
It follows that the range of $(I+i B_\oplus)\Psi^{**}$
belongs to $\sD=\cdom K$.
Now by Remark \ref{rem3.2} this implies
that for all $f\in \dom \Psi^{**}$ one has the equality
\begin{equation}\label{neweq00}
 (I+B_\oplus)(\Psi^{**})_{\s}f = C_0^{1/2}(I+iC)C_0^{1/2} P_\sD (\Psi^{**})_{\s}f.
 \end{equation}
This leads to
\[
 \Psi^* (I + i B_\oplus) \Psi^{**} = \Psi^* C_0^{1/2}(I+iC)C_0^{1/2} P_\sD (\Psi^{**})_{\s},
\]
which proves \eqref{HF2}. Similarly by substituting \eqref{neweq00} into the first formula
for $\st_F$ and noting that $P_\sD C_0^{1/2}=P_\sD$, one obtains the second formula in \eqref{tF2}.
\end{proof}

Also the construction of the Kre\u{\i}n extension for the sum $H_1+H_2$
can be found in \cite[Theorem~3.2]{HSS19}.
However, the corresponding form $\st_K$ was described
only under additional conditions to prevent the difficulty that appears by the fact
that the multivalued part of $(H_1+H_2)_K$ is in general not invariant under
the mapping $B_\oplus$. Theorem \ref{s-repr} allows a removal of these additional
conditions and leads to a description of the form $\st_K$ in the general situation.

For this purpose, decompose the Hilbert space $\sH\times \sH$ as follows
\begin{equation}\label{hilbertn}
\sH\times\sH=\cdom K \oplus \mul K^*,
\end{equation}
and let $P$ be the orthogonal projection onto $\cdom K$.
Moreover, decompose the selfadjoint operator $B_\oplus \in \mathbf{B}(\sH\times\sH)$
accordingly:
\begin{equation}\label{Bplusdec1}
 B_\oplus=\begin{pmatrix} B_{11}& B_{12}\\ B_{12}^*& B_{22}\end{pmatrix}: \,
 \begin{pmatrix} \cdom K \\ \mul K^* \end{pmatrix} \to
 \begin{pmatrix} \cdom K \\ \mul K^* \end{pmatrix}.
\end{equation}
Next define the operator $C_0 \in \mathbf{B}(\cdom K^*)$ by
\begin{equation}\label{c0plus}
C_0=I+B_{12}(I+B_{22}^2)^{-1}B_{12}^*,
\end{equation}
and the operator $C \in \mathbf{B}(\cdom K^*)$  by
\begin{equation}\label{cplus}
 C=C_0^{-\half}\left[B_{11}-B_{12}(I+B_{22}^2)^{-\half}B_{22}
 (I+B_{22}^2)^{-\half}B_{12}^*\right]C_0^{-\half},
\end{equation}
which is clearly selfadjoint.

\begin{theorem}\label{KVNext}
Let $H_{1}$ and $H_{2}$ be maximal sectorial relations in a Hilbert space $\sH$,
let $K$ be defined by \eqref{s-s-opK}, and let $C_0$ and $C$ be
given by \eqref{c0plus} and \eqref{cplus}, respectively.
Then the Kre\u{\i}n extension of $H_{1} + H_{2}$ has the expression
\[
 (H_{1} + H_{2})_K=K^{**}(I + i B_\oplus)K^{*}
 = ((K^*)_{\s})^{\times}\, C_0^{1/2}(I+i C) C_0^{1/2}\,(K^*)_{\s}.
\]
The closed sectorial form $\st_K$ associated with $(H_{1} + H_{2})_K$ is given by
\[
  \st_{K} [f,g] = ((I+i C) C_0^{1/2} (K^{*})_{\s} f,C_0^{1/2} (K^{*})_{\s} g),
   \quad f, \, g  \in \dom \st_{K} = \dom K^{*}.
\]
\end{theorem}

\begin{proof}
The first equality in the first statement is proved in \cite[Theorem~3.2]{HSS19}.
The second equality is obtained by applying Theorem \ref{s-repr} to the sectorial relation
$K^{**}(I + i B_\oplus)K^{*}$.

The statement concerning the form $\st_{K}$ is a consequence of
this second representation of $(H_{1} + H_{2})_K$, since $C_0^{1/2} (K^{*})_{s}$ is a closed operator
and hence one can apply Theorem \ref{s-repr0o} to get the desired expression for the corresponding form $\st_K$.
\end{proof}

 The form $\st_K$ described in Theorem \ref{KVNext} can be used to give a complete description of all \textit{extremal maximal sectorial extensions} of the sum $H_1+H_2$.
Namely, a maximal sectorial extension $\wt H$ of a sectorial relation $S$ is extremal precisely when the corresponding closed sectorial form
$\st_{\wt H}$ is a restriction of the closed sectorial form $\st_K$ generated by the Kre\u{\i}n extension $S_K$ of $S$;
see e.g. \cite[Definition 7.7, Theorems~8.2,~8.4,~8.5]{HSSW17}.
Therefore, Theorem \ref{KVNext} implies the following description of all extremal maximal sectorial extensions of $H_1+H_2$.

\begin{theorem}\label{s-sum-caracter}
Let $H_{1}$ and $H_{2}$ be
maximal sectorial relations in $\sH$, let $\Psi$ and $K$ be defined by \eqref{s-s-qu} and
\eqref{s-s-opK}, respectively, and let $P_\sD$ be the orthogonal projection from $\sH\times\sH$ onto $\sD=\cdom K$.
Then the following statements are equivalent:
\begin{enumerate}
\def\labelenumi{\rm (\roman{enumi})}

\item $\widetilde{H}$ is an extremal maximal sectorial extension of $H_{1}+H_{2}$;

\item $\widetilde{H} = R^{\ast} (I+i C ) R$, where $R$ is a closed linear operator satisfying
\[
  C_0^{1/2} P_\sD \Psi^{**} \subset   R \subset C_0^{1/2}(K^*)_\s.
\]
\end{enumerate}
\end{theorem}

\begin{proof}
 For comparison with the abstract results this statement will be proved by means of
the constructions used in \cite{HSSW17}.
Let $S=H_1+H_2$ then the sectorial relation $S$ gives rise to a Hilbert space $\sH_S$
and a selfadjoint operator $B_S \in \mathbf{B}(\sH_S)$ such that the Friedrichs extension $S_F$ and
the Kre\u{\i}n extension $S_K$ of $S$ are given by
\[
S_F=Q^*(I+iB_S)Q^{**}, \quad t_F=J^{**}(I+iB_S)J^*,
\]
with corresponding forms
\[
 \st_F[f,g]=((I+iB_S)Q^{**} f, Q^{**}g), \quad f,g \in \dom Q^{**},
\]
and
\[
\st_K[f,g]=((I+iB_S)J^{*} f, J^{*}g), \quad f,g \in \dom J^{*};
\]
see \cite[Theorem~8.3]{HSSW17}.
Here $Q: \sH \to \sH_S$ is an operator and $J: \sH_S \to \sH$ is a densely defined linear relation such that
\[
 J \subset Q^*, \quad Q \subset J^*;
\]
in particular, the adjoint $J^*$ is an operator.

Recall from Theorem \ref{KVNext} that
\[
  \st_{K} [f,g] = ((I+i C) C_0^{1/2} (K^{*})_{\s} f,C_0^{1/2} (K^{*})_{\s} g),
\]
while Theorem \ref{ss-twee} gives
\[
\st_{F} [f,g]= ((I + iB_\oplus ) \Psi^{**} f, \Psi^{**} g)
 = (C_0^{1/2}(I+iC)C_0^{1/2} P_\sD (\Psi^{**})_{s} f,P_\sD (\Psi^{**})_{s} g).
\]
Now apply \cite[Theorem 8.4]{HSSW17} and Corollary \ref{newnew}.
 \end{proof}

\subsection{The form sum construction}\label{sec3.3}

The maximal sectorial relations $H_{1}$ and $H_{2}$ generate the
following closed sectorial form
\begin{equation}
\label{s-s-fs}
 ((I+iB_{1})A_{1\s}^{\half} h, A_{1\s}^{\half} k)
 +((I+iB_{2})A_{2\s}^{\half} h, A_{2\s}^{\half} k),
 \quad h,k \in \dom A_{1}^{\half} \cap \dom A_{2}^{\half}.
\end{equation}
Observe that the restriction of this form to $\dom \Psi^{**}$ is
equal to
\begin{equation}
\label{s-sum-form}
 (\Psi^{**}h, \Psi^{**}k)
 =
 ((I+iB_{1})A_{1\s}^{\half} h, A_{1\s}^{\half} k)
 +((I+iB_{2})A_{2\s}^{\half} h, A_{2\s}^{\half} k),
 \quad h,k \in \dom \Psi^{**},
\end{equation}
since $\Psi^{**} \subset (\Phi^*)_\s$, cf. \eqref{s-s-qus}.
Thus, the form in \eqref{s-s-fs} has a natural domain
which is in general larger than $\dom \Psi^{**}$.

\begin{theorem}
\label{s-sum-een} Let $H_{1}$ and $H_{2}$
be maximal sectorial relations in $\sH$,
let $\Phi$ be given by \eqref{s-s-Einz},
and let $\sE=\clos \sE_0$ and $\sF=\clos \sF_0$
be defined by \eqref{E0} amd \eqref{F0}.
Then the maximal sectorial relation
\[
 \Phi^{**} (I+iB_\oplus ) \Phi^*
\]
is an extension of the relation $H_{1}+H_{2}$, which corresponds to
the closed sectorial form in \eqref{s-s-fs}.

Moreover, the following statements are equivalent:
\begin{enumerate}
\def\labelenumi{\rm (\roman{enumi})}

\item $\Phi^{**} (I+iB_\oplus ) \Phi^*$ is extremal;

\item $\sE=\sF$.
\end{enumerate}
\end{theorem}

\begin{proof}
The first statement is proved in \cite[Theorem 3.5]{HSS19}.
For the proof of the equivalence of (i) and (ii) appropriate modifications are needed
in the arguments used in the proof of \cite[Theorem 3.5]{HSS19}.
The special case treated there was based on the additional assumption that $\sD=\sE$,
where $\sD=\cdom K$; a condition which implies the invariance of $\mul K^*$
under the operator $B_\oplus$. In the present general case such an invariance property cannot be assumed.
Now for simplicity denote the form sum extension of $H_1+H_2$ briefly by $\wh H=\Phi^{**} (I+iB_\oplus ) \Phi^*$.

(i) $\Rightarrow$ (ii)
Assume that $\wh H$ is extremal.
Since $\sE \subset \sF$ by \eqref{s-s-ef}, it is enough to prove
the inclusion $\sF \subset \sE$.
By Theorem \ref{s-sum-caracter} and $\mul \Phi^*=\mul H_{1} \times \mul H_{2}$
one sees
\begin{equation}
\label{NICE1}
 \wh H=((\Phi^{*})_{\s})^{*} (I+iB_\oplus ) (\Phi^{*})_{\s} = R^{\ast} (I+i C ) R,
\end{equation}
for some closed operator $R$ satisfying
\begin{equation}\label{Rincl}
  C_0^{1/2} P_\sD \Psi^{**} \subset   R \subset C_0^{1/2}(K^*)_\s,
\end{equation}
where $P_\sD$ is the orthogonal projection of $\sH \times \sH$ onto $\sD=\cdom K$.
Recall that $(\Phi^{*})_{\s} \subset \Phi^{*} \subset
K^{*}$ and hence $P_\sD(\Phi^{*})_{\s}\subset P_\sD K^*=(K^*)_{\s}$.
Moreover, one has $\dom P_\sD(\Phi^{*})_{\s} =\dom (\Phi^{*})_{\s} = \dom R$,
since by assumption these two domains coincide with the corresponding joint form domain.
Denoting $\wh R= C_0^{-1/2}R$, one has $\dom \wh R=\dom P_\sD(\Phi^{*})_{\s}$ and \eqref{NICE1} can be rewritten as
\begin{equation}
\label{NICE2}
 \wh H=((\Phi^{*})_{\s})^{*} (I+iB_\oplus ) (\Phi^{*})_{\s} = \wh R^{\ast}C_0^{1/2}(I+iC)C_0^{1/2}\wh R,
\end{equation}
where $\wh R$ satisfies $P_\sD \Psi^{**} \subset  \wh R \subset (K^*)_\s$.
One concludes that $P_\sD(\Phi^{*})_{\s}=\wh R$, since both operators are restrictions of $(K^*)_\s$, and thus
\begin{equation}\label{Rhat*}
 ((\Phi^{*})_{\s})^* P_\sD = \wh R^*.
\end{equation}
Now one obtains from \eqref{NICE2} the equalities
\begin{eqnarray}
 ((\Phi^{*})_{\s})^{*} (I+iB_\oplus ) (\Phi^{*})_{\s} & = & \wh R^{\ast}C_0^{1/2}(I+iC)C_0^{1/2}\wh R
 \nonumber \\
 & = & ((\Phi^{*})_{\s})^{*}P_{\sD} C_0^{1/2}(I+iC)C_0^{1/2}\wh R
 \nonumber \\
 & = & ((\Phi^{*})_{\s})^{*} C_0^{1/2}(I+iC)C_0^{1/2}\wh R . \nonumber
\end{eqnarray}
Hence, for every $f\in\dom \wh H$ one has
$$(I+iB_\oplus )(\Phi^{*})_{\s}f-C_0^{1/2}(I+iC)C_0^{1/2}\wh R f \in \ker ((\Phi^{*})_{\s})^{*}.$$
Here $C_0^{1/2}(I+iC)C_0^{1/2}\wh R f \in \sD=\cdom K$ and $\sD=\cdom K=(I+iB_\oplus) \sE$; see \eqref{sDsE}.
Therefore, there exists $\varphi\in \sE$ such that $C_0^{1/2}(I+iC)C_0^{1/2}\wh R f=(I+iB_\oplus)\varphi$.
On the other hand, $(\Phi^{*})_{\s}f \in \sF=\cran(\Phi^{*})_{\s}=(\ker ((\Phi^{*})_{\s})^{*})^\perp$, see \eqref{s-s-qus}, \eqref{F0}.
Since $\varphi\in \sE\subset \sF$, this yields
\[
 \left((I+iB_\oplus )((\Phi^{*})_{\s}f-\varphi ), (\Phi^{*})_{\s}f-\varphi \right)=0,
\]
and thus $(\Phi^{*})_{\s}f-\varphi=0$. Consequently, for all $f\in \dom \wh H$ one has
\[
 (\Phi^{*})_{\s}f\in \sE.
\]
Since $\dom \wh H$  is a core for the corresponding closed form, or equivalently, the closure of
$(\Phi^{*})_{\s} \uphar\dom \wh H$ is equal to
$(\Phi^{*})_{\s}$, the claim follows:
$\sF = \cran (\Phi^{*})_{\s} \subset \sE.$

(ii) $\Rightarrow$ (i) Assume that $\sE = \sF$.
Then $\sF_{0}=\ran (\Phi^{*})_{\s}\subset \sE$ and hence for all
$f\in \dom (\Phi^{*})_{\s}$ one has $(I+B_\oplus)(\Phi^{*})_{\s}f \in \cdom K$.
By Remark \ref{rem3.2} this implies that
\begin{equation}\label{neweq0}
 (I+B_\oplus)(\Phi^{*})_{\s}f = C_0^{1/2}(I+iC)C_0^{1/2} P_\sD (\Phi^{*})_{\s}f.
 \end{equation}
On the other hand, as shown above $P_\sD(\Phi^{*})_{\s}\subset P_\sD K^*=(K^*)_{\s}$.
Let $\wh R$ be the closure of $(K^*)_{\s}\uphar\dom (\Phi^{*})_{\s}$.
Then $\wh R^*$ satisfies the identity \eqref{Rhat*}.
Since $\Psi^{**}\subset (\Phi^*)_\s$ (see \eqref{Psiclos}) one obtains $P_\sD\Psi^{**}\subset \wh R$.
The identities \eqref{Rhat*} and \eqref{neweq0}
imply that for all $f\in \dom \wh H$ the equalities
\[
\begin{split}
   ((\Phi^{*})_{\s})^{*} (I+B_\oplus)(\Phi^{*})_{\s}f
   &  = ((\Phi^{*})_{\s})^{*} P_\sD C_0^{1/2}(I+iC)C_0^{1/2} P_\sD (\Phi^{*})_{\s} f \\
   &  = \wh R^* C_0^{1/2}(I+iC)C_0^{1/2} \wh R f
\end{split}
\]
hold. Then the closed operator $R=C_0^{1/2}\wh R$ satisfies the inclusions \eqref{Rincl}
as well as the desired identity
$((\Phi^{*})_{\s})^{*} (I+B_\oplus)(\Phi^{*})_{s}=R^*(I+iC)R$,
and thus $\wh H$ is  extremal,
cf. Theorem \ref{s-sum-caracter}.
\end{proof}

Theorem \ref{s-sum-een} is a generalization of \cite[Theorem 3.5]{HSS19}, where
an additional invariance of $\mul K^*$ under the operator $B_\oplus$ was used.
Moreover, Theorem \ref{s-sum-een} generalizes a corresponding result for
the form sum of two closed nonnegative forms established earlier in \cite[Theorem~4.1]{HSSW2007}.

The present result relies on Theorem \ref{s-repr}, where the description of
the closed sectorial form generated by a general maximal sectorial relation
of the form $H=T^*(I+iB)T$ where $T$ is a closed relation. This generality implies that
with special choices of $B$ the relation $H$ can be taken to be nonnegative and selfadjoint, i.e.,
the corresponding closed form $\st$ becomes nonnegative; see Example \ref{examp}.

\end{document}